\documentclass[12pt]{article}
\usepackage{fullpage,graphicx,psfrag,amsmath,amsfonts,verbatim,url,amssymb,amsthm}
\usepackage[small,bf]{caption}
\usepackage[font=footnotesize]{subcaption}
\usepackage[noend]{algpseudocode}
\usepackage[T1]{fontenc}
\usepackage{tikz}
\usetikzlibrary{positioning,calc}
\definecolor {processblue}{cmyk}{0.96,0,0,0}
\usepackage{graphicx}
\usepackage{color}
\usepackage[ruled,vlined, linesnumbered, algo2e]{algorithm2e}
\usepackage{booktabs}
\usepackage[para,online,flushleft]{threeparttable}
\usepackage[autostyle]{csquotes}
\usepackage{supertabular}
\usepackage{algorithmicx}
\usepackage{amsmath}
\usepackage{multirow}
\usepackage{pifont}
\usepackage{lscape}

\usepackage[singlelinecheck=off]{caption}
\usepackage{hyperref}
\hypersetup{
	colorlinks   = true,
	citecolor   = blue
}
\theoremstyle{definition}

\newcommand{\cmark}{{\color{blue}\ding{51}}}%
\newcommand{\xmark}{{\color{red}\ding{55}}}%

\newcommand{\minimizing}[1]{\underset{#1}{\operatorname{Minimizing}}\;}

\algrenewcommand\alglinenumber[1]{
    {\sf\footnotesize\addfontfeatures{Colour=888888,Numbers=Monospaced}#1}}
\algrenewcommand\algorithmicrequire{\textbf{Precondition:}}
\algrenewcommand\algorithmicensure{\textbf{Postcondition:}}
\mathchardef\mhyphen="2D
\renewcommand{\algorithmicrequire}{\textbf{Input:}}
\newcommand*\samethanks[1][\value{footnote}]{\footnotemark[#1]}

\newcommand{\BA}{\begin{array}}
	\newcommand{\EA}{\end{array}}

\newcommand{\argmin}{\mathop{\rm argmin}}

\newcommand{\argmax}{\mathop{\rm argmax}}

\newtheorem{theorem}{Theorem}[section]

{\begin{quote}}{\end{quote}}

\makeatletter
\long\def\@makecaption#1#2{
	\vskip 9pt
	\begin{small}
		\setbox\@tempboxa\hbox{{\bf #1:} #2}
		\ifdim \wd\@tempboxa > 5.5in
		\begin{center}
			\begin{minipage}[t]{5.5in}
				\addtolength{\baselineskip}{-0.95pt}
				{\bf #1:} #2 \par
				\addtolength{\baselineskip}{0.95pt}
			\end{minipage}
		\end{center}
		\else
		\hbox to\hsize{\hfil\box\@tempboxa\hfil}
		\fi
	\end{small}\par
}
\makeatother

\newcounter{oursection}

\newcounter{lecture}

\title{Accelerated Parallel and Distributed Algorithm using Limited Internal Memory for Nonnegative Matrix Factorization}
\author{Duy Khuong Nguyen
		~\thanks{Japan Advanced Institute of Science and Technology, Japan}
		~\thanks{University of Engineering and Technology, Vietnam National University, Hanoi, Vietnam}
	\and Tu-Bao Ho
		~\samethanks[1]
		~\thanks{John von Neumann Institute, Vietnam National University, Ho Chi Minh City, Vietnam}}

\begin{document}
\maketitle

\begin{abstract}
Nonnegative matrix factorization (NMF) is a powerful technique for dimension reduction, extracting latent factors and learning part-based representation. For large datasets, NMF performance depends on some major issues: fast algorithms, fully parallel distributed feasibility and limited internal memory. This research aims to design a fast fully parallel and distributed algorithm using limited internal memory to reach high NMF performance for large datasets. In particular, we propose a flexible accelerated algorithm for NMF with all its $L_1$ $L_2$ regularized variants based on full decomposition, which is a combination of an anti-lopsided algorithm and a fast block coordinate descent algorithm. The proposed algorithm takes advantages of both these algorithms to achieve a linear convergence rate of $\mathcal{O}(1-\frac{1}{||Q||_2})^k$ in optimizing each factor matrix when fixing the other factor one in the sub-space of passive variables, where $r$ is the number of latent components; where $\sqrt{r} \leq ||Q||_2 \leq r$. In addition, the algorithm can exploit the data sparseness to run on large datasets with limited internal memory of machines. Furthermore, our experimental results are highly competitive with 7 state-of-the-art methods about three significant aspects of convergence, optimality and average of the iteration number. Therefore, the proposed algorithm is superior to fast block coordinate descent methods and accelerated methods. 

\vspace{0.5cm}
\textbf{Keywords}: Non-negative matrix factorization, Anti-lopsided algorithm, Cooridinate descent algorithm, and Parallel and distributed algorithm.
\end{abstract}

\clearpage
\tableofcontents
\clearpage

\section{Introduction}

Nonnegative matrix factorization (NMF) is a powerful technique widely used in applications of data mining, signal processing, computer vision, bioinformatics, etc.~\cite{Zhang2011a}. Fundamentally, NMF has two main purposes. First, it reduces dimension of data making learning algorithms faster and more effective as they often work less effectively due to the curse of dimensionality \cite{Helen2005}. Second, NMF helps extracting latent components and learning part-based representation, which are the significant distinction from other dimension reduction methods such as Principal Component Analysis (PCA), Independent Component Analysis (ICA), Vector Quantization (VQ), etc. This feature originates from transforming data into lower dimension of latent components and non-negativity constraints \cite{Donoho2004,Gillis2014ArXiv,Lee1999}. 

In the last decade of fast development, there were remarkable milestones. The two first milestones in early days of the NMF historical development were its mathematical formulations as positive matrix factorization with Byzantine algorithms~\cite{Paatero1994} and as parts-based representation with a simple effective algorithm~\cite{Lee1999}. The last decade has witnessed the rapid NMF development~\cite{Zhang2011a,Wang2013}. Various works on NMF can be viewed in three major perspectives: variants of NMF, algorithms and applications. In particular, variants of NMF are based on either divergence functions~\cite{Seung2001,Zhang2011}, or constraints~\cite{Hoyer2004,Pascual2006}, or regularizations~\cite{Choi2008,Li2007non}. Most NMF algorithms were developed along two main directions: geometric greedy algorithms~\cite{Thurau2011} and iterative multiplicative update algorithms. Although geometric greedy algorithms are usually fast, they are hard to trade off complexity, optimality, loss information and sparseness.

More recently, it is well recognized that the most challenging problems in iterative multiplicative update algorithms for NMF are fast learning, limited internal memory, parallel distributed computation, among others. In particular, fast learning is essential in learning NMF models from large datasets, and it is indeed difficult to carry out them when the number of variables is very large. In addition, the limited internal memory is one of the most challenging requirements for big data~\cite{Guan2013}, because data has been exploring rapidly while the internal memory of nodes is always limited. Finally, parallel and distributed computation makes NMF applications feasible for big data~\cite{Liu2010}.

\begin{table}
	\centering
	\begin{threeparttable}
		\setlength\tabcolsep{0.5 pt}
		\caption{Comparison Summary of NMF solvers}
		\begin{tabular}{|l|c|c|c|c|c|c|c|c|c|c|}
			\hline
			\multirow{2}{*}{\textbf{Criteria}} & \multicolumn{2}{c|}{\textbf{Inexact}}& \multicolumn{4}{c|}{\textbf{Exact}} &\multicolumn{4}{c|}{\textbf{Accelerated}} \\ \cline{2-11}
			& \textbf{MUR} & \textbf{PrG}  & \textbf{Qn} & \textbf{Nt} & \textbf{AcS} & \textbf{BlP} &\textbf{FCD} & \textbf{AcH} & \textbf{Ne} & \textbf{Alo} \\ 
			\hline
			\textbf{Guaranteed Convergence}  & \xmark & \xmark & \xmark &  \xmark &  \xmark &  \xmark & \xmark & \xmark & \cmark$\frac{1}{k^2}$ & \cmark$(1-\frac{1}{||Q||_2})^k$ \\ \hline
			\textbf{Exploit Data Spareness}  & \xmark & \xmark & \xmark & \xmark & \xmark & \xmark & \cmark & \cmark & \xmark & \cmark \\ \hline
			\textbf{Limited Internal Memory}  & \multicolumn{9}{c|}{$\mathcal{O}(mn+r(r+n+m))$}  & $\mathcal{O}(r(r+n))$ \\ \hline
			\textbf{Fully Parallel \& Distributed}  & \xmark & \xmark & \xmark & \xmark & \xmark  & \xmark & \xmark & \xmark & \xmark & \cmark \\ \hline
			\textbf{Optimization Problem Size}  & \multicolumn{3}{c|}{$r(m,n)$} & $r$ & \multicolumn{3}{c|}{$r(n, m)$}  & $(m,n)$ &  $r(m,n)$ & $r$ \\ 
			\hline
		\end{tabular}
		\begin{tablenotes}
			\small
			\item \cmark means considered, and \xmark means not considered\\
			\item $\sqrt{r} \leq ||Q||_2 \leq r$, $n \times m$ is the data matrix size, $r$ is the number of latent components\\
			\item $(m, n)=\text{max}(m, n)$, and $r(m, n)=r.\text{max}(m, n)$\\
			\item Abbreviations: \textbf{MUR}: Multiplicative Update Rule~\cite{Lee1999};
			\textbf{PrG}: Projected Gradient methods~\cite{Lin2007Pro}; 
			\textbf{Nt}: Newton-type methods~\cite{Kim2007};
			\textbf{Qn}: Projected Quasi-Newton~\cite{Zdunek2006}; 
			\textbf{AcS}: Fast Active-set-like method~\cite{Kim2008}; 
			\textbf{BlP}:  Block Principal Pivoting method~\cite{Kim2008Toward};
			\textbf{FCD}: Fast Coordinate Descent methods with variable selection~\cite{Hsieh2011};
			\textbf{AcH}:  Accelerated Hierarchical Alternating Least Squares~\cite{Gillis2012HALS};
			\textbf{Nev}: Nesterov's optimal gradient method~\cite{Guan2012};
			\textbf{Alo}: The proposed method.
		\end{tablenotes}
		\label{tab:Algorithms}
	\end{threeparttable}
\end{table}

To deal with these challenges, this work develops an accelerated algorithm for NMF and its $L_1$ $L_2$  regularized variants having several major advantages that are summarized in Table~\ref{tab:Algorithms}. In this paper, we contribute five folders as follows:

\indent $\bullet$ \textit{NMF and its variants}: We fully decompose NMF and its $L_1$ $L_2$ regularized variants into non-negative quadratic programming problems. This decomposition makes the proposed algorithm flexible to adapt all  $L_1$ $L_2$ regularized NMF in an unified framework that can trade-off the quality of information loss, sparsity and smoothness.

\indent $\bullet$ \textit{Algortihm}: We employ a combinational algorithm of an anti-lopsided algorithm and a fast block coordinate descent algorithm for non-negative quadratic programming. The algorithm reduces variable scaling problems to achieve linear convergence rate of $(1-\frac{1}{||Q||_2})^k$ in optimizing each factor matrix in the sub-space of passive variables, which is advanced to fast coordinate methods and accelerated methods in terms of efficiency as well as convergence rate. In addition, the size of optimization problem is reduced into $r$ ($r \ll m, n$), which is the smallest among the state-of-the-art methods. Hence, the algorithm has the low complexity and converges very fast to the optimal solution, and it is highly potential to be applied in alternating least squares methods for factorization models.

\indent $\bullet$ \textit{Parallel and Distribution}: The proposed algorithms are fully parallel and distributed on limited internal memory systems, which is crucial for big data when computing nodes having limited internal memory that cannot hold the whole dataset.   

\indent $\bullet$ \textit{Implementation}:  The proposed algorithms are convenient to implement for hybrid multi-core distributed systems because  this algorithm works on each individual instance and each latent feature.

\indent $\bullet$ \textit{Comparision}: This is the first time that state-of-the-art algorithms in different research directions for NMF are compared together.

The rest of paper is organized as follows: Section \ref{sec:Background} discusses the background and related works of NMF; Section \ref{sec:OurAlgorithm} mentions our proposed algorithm; Section \ref{sec:Complexity} gives a complexity analysis of our proposed algorithms; Section \ref{sec:Experimental} experimentally compares our proposed algorithm with state-of-the-art algorithms for NMF among remarkable approaches; our conclusion is stated in Section \ref{sec:Conclusion}.

\section{Background and Related Works}
\label{sec:Background}

\subsection{Background}

Mathematically, NMF in Frobenius norm is defined as follows:

\noindent \textbf{Definition 1 [NMF]:} {\it Given a dataset consisting of $m$ vectors in a $n$-dimension space $V = [V_1, V_2,$ $..., V_m] \in R^{n \times m}_+$, where each vector presents a data instance. NMF seeks to decompose $V$ into a product of two nonnegative factorizing matrices $G$ and $F$, where $G = [G_1, ..., G_r] \in R^{n \times r}_+$ and $F = [F_1, ..., F_m] \in R^{r \times m}_+$ are the latent component matrix and the coefficient matrix respectively, $V \approx GF$, in which the quality of approximation can be guaranteed by the objective function in Frobenius norm: $D(V||GF) = ||V-GF||_2^2$.
}

Although NMF is a non-convex problem, optimizing each factor matrix when fixing the other one is a convex problem. In other words, $F$ can be traced when $G$ is fixed, and vice versa. Furthermore, $F$ and $G$ have different roles although they are symmetric in the objective function. $G$ are latent components to represent data instances $V$ by coefficients $F$. Hence, NMF can be considered as a latent factor model of latent components $G$, and learning this model is equivalent to find out latent components $G$. Therefore, in this paper, we propose an accelerated parallel and distributed algorithm to learn NMF models $G$ for large datasets.

\subsection{Related works}

NMF algorithms can be divided into two groups: the greedy algorithms and the iterative multiplicative update algorithms. The greedy algorithms~\cite{Thurau2011} are often based on geometric interpret-ability, and they can be extremely fast to deal with large datasets. However, it is hard to trade off complexity, optimality, loss information and sparseness. The iterative multiplicative update algorithms such as \enquote{two-block coordinate descent} often consist of two steps, each of them fixes one of two matrices to replace the other matrix for obtaining the convergence of the objective function. There are numerous studies on these algorithms, see Table~\ref{tab:Algorithms}, because NMF is nonconvex, though two steps corresponding to two non-negative least square (NNLS) sub-problems are convex~\cite{Guan2012,Kim2014}. In addition, various constraints and optimization strategies have been used to trade off the convexity,  information loss, complexity, sparsity, and numerical instability. 

Based on the optimization updating strategy, these iterative multiplicative update algorithms can be further divided into three sub-groups:

$\bullet$ \textit {Inexact Block Coordinate Descent}: The algorithms' common characteristic is their usage of gradient methods to seek an approximate solution for NNLS problems, which is neither optimal nor fulfilling of fast approximations and accelerated conditions. Lee {\em et al.} \cite{Lee1999} proposed the (basic) NMF problem and simple multiplicative updating rule (MUR) algorithm using first-order gradient method to learn the part-based representation. Seung {\em at al.} \cite{Seung2001} concerned rescaling gradient factors with carefully selected learning rate to achieve a faster convergence rate. Subsequently, Lin~\cite{Lin2007Con} modified MUR, which is theoretically proved getting a stationary point (a local minimum optimization). However, that algorithm cannot improve the convergence rate. Berry {\em et al.}~\cite{Berry2007} projected nonnegative least square (PNLS) solutions into nonnegative quadratic space by setting negative entries in the matrices to zero. Although this algorithm does not guarantee the convergence, it is widely applied in real applications.
In addition, Bonettini {\em et al.}~\cite{Bonettini2011} used line search based on Amijo rule to obtain better solutions for matrices. Theoretically, this method can achieve optimal solutions for factor matrices as exact block coordinate descent group, but it very slowly tends to  stationary points because the line search is time-consuming.

$\bullet$ \textit {Exact Block Coordinate Descent}: In contrast to the first sub-group, the common characteristic in this group is obtaining optimal solutions for two NNLS problems in each iteration. Zdunek {\em et al.}~\cite{Zdunek2006} employed second-order quasi-Newton method with inverse of Hessian matrix to estimate the step size, aiming to a faster convergence than projected methods. However, this algorithm may be slow and non-stable because of the line search. Subsequently, Kim {\em et al.}~\cite{Kim2007} used rank-one to Broyden-Fletcher-Goldfarb-Shanno (BFGS) algorithm to approximate the inverse of Hessian matrix.  Furthermore, Chih-Jen Lin~\cite{Lin2007Pro} proposed several algorithms based on projected gradient methods and exact line search. Theoretically, this method can obtain more accurate solutions, however it is time-consuming because of exact line search and the number of iterations increased by the large number of variables. Moreover, Kim {\em et al.}~\cite{Kim2008,Kim2008Toward} proposed two active-set methods based on Karush-Kuhn-Tucker (KKT) conditions, in which the variables are divided into two sets: a free set and an active set. Only the free set contains variables that can optimize the objective functions. Removing the number of redundant variables makes their algorithms improve the convergence rate significantly. However, the method still has heavy computation for large-scale problems. 

$\bullet$ \textit {Accelerated Block Coordinate Descent}: The accelerated methods use fast solution approximations satisfying accelerated conditions to reduce the complexity and to keep fast convergence. The accelerated conditions are different constraints in different methods to guarantee convergence to the optimal solution in comparison with the initial value. These accelerated methods are developed due to the limitation of inexact methods having slow convergence, and exact methods having high complexity in each iteration. Particularly, for inexact methods, they have slow convergence because of the high complexity of solution approximations in each iteration or a large number of iterations that leads to the highly expensive computation between two sequential iterations. Furthermore, the exact methods have high complexity in each iteration, however obtaining optimal solutions in every iteration is controversial because it can lead to zig-zag problems when optimizing a non-convex function of two independent sets of variables.

Firstly, Hsieh {\em et al.}~\cite{Hsieh2011} proposed a fast coordinate descent method with the best variable selection to reduce the objective function. The algorithm iteratively selects variables to update the approximate solution until the accelerated stopping condition $max_{ij} D^G_{ij} < \epsilon p_{\text{init}}$ satisfied, where $D^G_{ij}$ is the reduction of the objective function based on the variable $G_{ij}$, and $p_{\text{init}}$ is the maximum initial reduction over the matrix $G$.
Although the greedy update method does not have guaranteed convergence, it has the fast convergence speed in many reports.

Subsequently, Gillis and Glineur~\cite{Gillis2012HALS} proposed a number of accelerated algorithms using fast approximation by fixing all variables but excepting a single column of factor matrices. This framework improved significantly the effectiveness of multiplicative updates~\cite{Lee2001}, hierarchical alternating least squares (HALS) algorithms~\cite{Cichocki2007} and projected gradients \cite{Lin2007Pro}. These algorithms achieve the accelerated condition in each iteration such as that $||G^{(k,l+1)}-G^{(k,l)}||^2_2 \leq \epsilon ||G^{(k,1)}-G^{(k,0)}||^2_2$ is the stopping condition when optimizing the objective function on $G$ if fixing $F$. Although these greedy algorithms does not have guaranteed convergence, their results are highly competitive with the inexact and exact methods. 

More recently, Guan {\em et al.}~\cite{Guan2012} employed Nesterov's optimal methods to optimize NNLS with fast convergence rate $\mathcal{O}(1/k^2)$ to achieve the accelerated convergence condition $||\frac{\partial f}{\partial G_{(k, l+1)}}||^2_2 \leq \epsilon ||\frac{\partial f}{\partial G_{(k, 0)}}||^2_2$. Although   Guan {\em et al.}'s method~\cite{Guan2012} has a fast convergence rate $\mathcal{O}(1/k^2)$, it has several drawbacks such as working on the whole factor matrices, and less flexibility for regularized NMF variants. Furthermore, this approach does not consider the issues of parallel and distribution, and they require numerous iterations to satisfy the accelerated condition because the step size is limited by $\frac{1}{L}$, where $L$ is Lipschitz constant. 

To deal with the above issues of accelerated methods, in next section, we propose an accelerated parallel and distributed algorithm for NMF and its regularized $L_1\ L_2$ variants with linear convergence in optimizing each factor matrix when fixing the other factor one.

\section{Proposed Algorithm}
\label{sec:OurAlgorithm}

To read easily, this section hierarchically presents our proposed algorithm. First, an iterative multiplicative update accelerated algorithm is introduced. Then, a transformational technique fully decomposes the objective functions of NMF into basic computation units as nonnegative quadratic programming (NQP) problems. After that, a modified version of the algorithm is proposed to deal with the issues of parallel and distributed systems. Subsequently, a combinational method of an anti-lopsided algorithm and a fast coordinate descent algorithm is developed to effectively solve NQP problems. Finally, extensions for $L_1 L_2$ regularized NMF is discussed.

\subsection{Iterative multiplicative update accelerated algorithm}

For solving NMF, we employ an iterative multiplicative update accelerated algorithm, like expectation-maximization (\textit{EM}) algorithm, presented in Algorithm~\ref{algo:IMU}. This algorithm consists of two main steps: one for finding $F^+$ ($F^+$ is updated $F$ in the iteration) when fixing $G$ and the other for finding $G^+$ when fixing $F$. In the first step called \textit{E-step}, we find $F^+$, each column of which $F^+_i$ is the new representation of a data instance $V_i$ in the new space of latent components $G$. Meanwhile, the other one, called \textit{M-step}, learns new latent components.

\begin{algorithm2e}
	\KwIn{Data matrix $V=\{V_i\}_{i=1}^m \in R^{n \times m}_+$ and $r$.}
	\KwOut{Latent components $G=\{G_k\}_{k=1}^r$.}
	\Begin{
		Randomize $r$ nonnegative latent components $\in R_+^{n \times r}$ \;
		\Repeat{Convergence condition is satisfied}{
			\textbf{E-step:} Fixing $G$ to find $F^+$ such that the accelerated condition is satisfied\;
			\textbf{M-step:} Fixing $F$ to find $G^+$ such that the accelerated condition is satisfied\;
		}
	}
	\caption{Iterative Multiplicative Update Accelerated Algorithm}\label{algo:IMU} 
\end{algorithm2e}

\subsection{Full decomposition for NMF}

This section discusses decomposing the objective function of NMF into non-negative quadratic programming (NQP) problems, which aims to fully parallelize and distribute the NMF computation. Particularly, in Algorithm \ref{algo:IMU}, the E-step is to find new coordinates of data instances in the space of latent components $G$ by minimizing $J(V || GF) = ||V - GF||^2_2 = \sum_{i=1}^m ||V_{i} - GF_i||^2_2$. Hence, minimizing $J(V|| GF)$ is equivalent to independently minimizing $||V_{i} - GF_i||^2_2$  for each instance $i$ since $G$ is fixed. Similarly, the M-step is also equivalent to independently minimizing $||V^T_{j} - F^TG^T_j||^2_2$  for each feature $j$, where $F$ is fixed. Hence, the basic computation units are nonnegative least-squares (NNLS) problems~\cite{Lawson1974}. 

For large datasets $n, m \gg r$, we equivalently turn these problems into nonnegative quadratic programmings (NQP):

\begin{equation}\label{eq:NNLS}
\begin{aligned}
& \underset{x}{\text{minimize}}
&& \frac{1}{2}||Ax - b||^2_2 \\
&\text{subject to}
&& x \succeq 0 \in R^r\\
&\text{where}
&& A \in R_+^{nr}, b \in R_+^n
\end{aligned}
\end{equation}
equivalent to
\begin{equation}\label{eq:QP}
\begin{aligned}
& \underset{x}{\text{minimize}}
& & f(x) = \frac{1}{2}x^THx + h^Tx \\
& \text{subject to}
& & x \succeq 0.\\
& \text{where}
& & H = A^TA,  h = -A^Tb\\
\end{aligned}
\end{equation}

\vspace{0.2cm}

Hence, finding new coefficients $F^+$ and new latent components $G^+$ can be fully paralleled and distributed into basic computation units as solving NQP problems. 

\subsection{Parallel and distributed algorithm using limited internal memory}
\label{subsec:ParallelandDistributed}

In this section, we design a parallel and distributed algorithm using limited internal memory for learning NMF model $G$, see Fig.~\ref{fig:DisSysDia}, which is  a modified version of Algorithm~\ref{algo:IMU}.

For large datasets, the computation can be untimely performed in a single process, so parallel and distributed algorithm environments are employed to speed up the computation. For parallel and distributed systems, we often face two major issues: dependency of computation units and limited internal memory computing nodes. In particular, computation units must be independently conducted  as much as possible, since any dependency of computing elements will increase the complexity of implementation and the delay of data transfer over the network that reduces the performance of system. Furthermore, for these parallel distributed systems, computation units are executed on computing nodes within a limited internal memory. In addition, accessing external memory will increase the complexity and reduce the performance. 

\begin{algorithm2e}
	\KwIn{Data matrix $V=\{V_m\}_{i=1}^m \in R^{nm}_+$ and $r$.}
	\KwOut{Latent components $G=\{g_k\}_{k=1}^r$.}
	\Begin{
		Randomize $r$ nonnegative latent components $G$ $\in R_+^{nr}$\;
		\Repeat{Convergence condition is satisfied}{
			$Y = 0 \in R^{nr}$ /* $Y = FV^T$*/\;
			$H = 0 \in R^{rr}$  /* $H = FF^T$ /*\;
			$Q = GG^T$\;
			$maxStop = 0$\;
			/*Parallel and distributed*/\\
			\For{$i$ = $1$ to $m$}{
				/*call Algorithm \ref{algo:NQP}*/\;
				$F_i = \minimizing{x \in R^{r} \succeq 0} (x^TQx/2 - V_m^TG^Tx)$\;
				$Y = Y + F_iV^T_i$\;
				$H = H + F_iF^T_i$\;
			}
			/*Parallel and distributed*/\\
			\For{$j$ = $1$ to $n$}{
				/*call Algorithm  \ref{algo:NQP}*/\;
				$G_j = \minimizing{x \in R^{r} \succeq 0} (x^THx/2-{Y_n}^Tx)$ \;
			}
		}
	}
	\caption{Parallel and Distributed Algorithm}\label{algo:Parallel} 
\end{algorithm2e}

\begin{figure}
	\centering
	\includegraphics[scale=0.6]{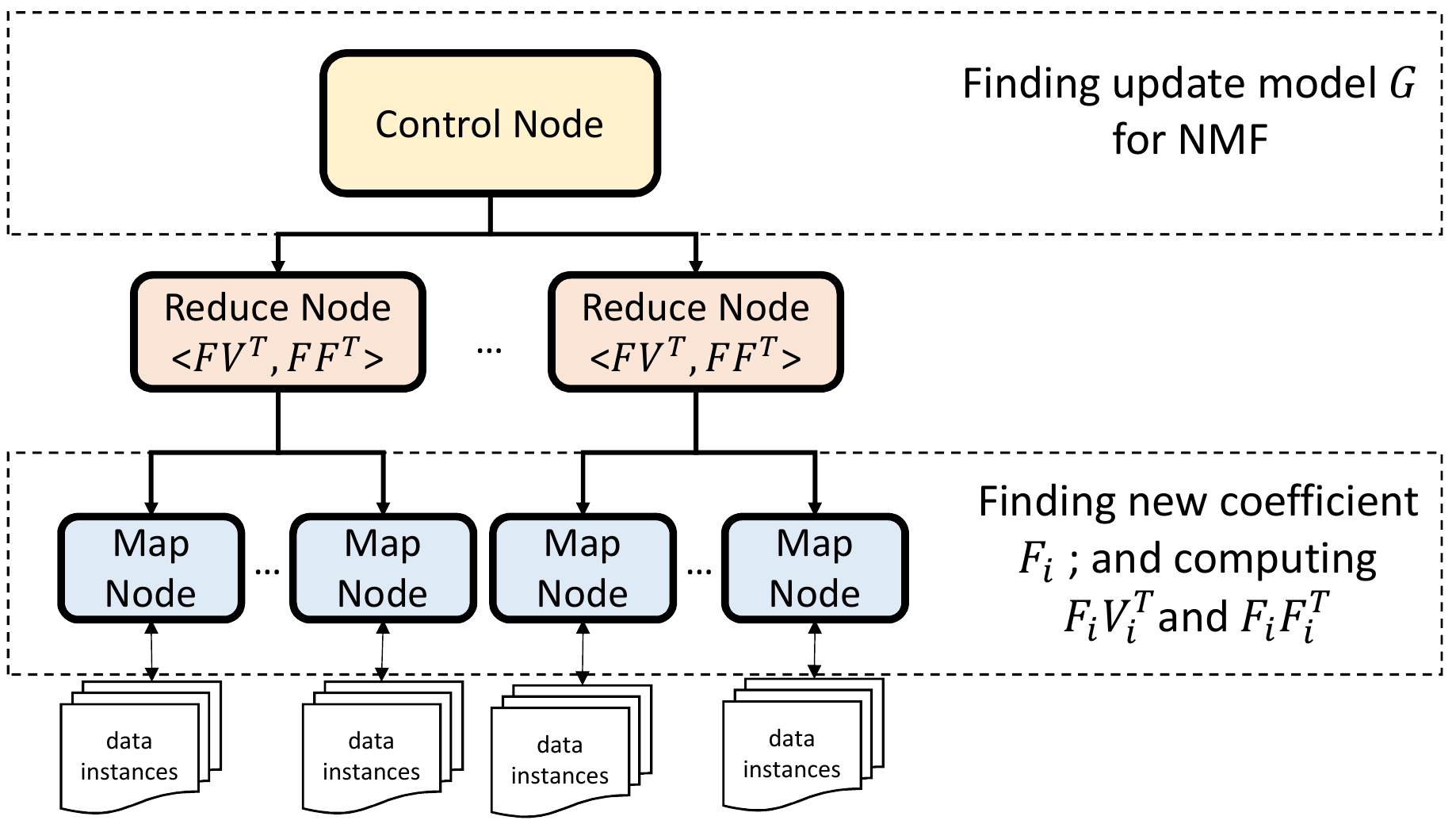}
	\caption{Distributed System Diagram for NMF}
	\label{fig:DisSysDia}
\end{figure}

For our proposed approach, the computation can be fully paralleled and distributed, and use limited internal memory in computing nodes  because the objective function is properly decomposed  to NQP problems. Particularly, Algorithm~\ref{algo:Parallel} presents a  modified version of  iterative multiplicative update algorithms, in which computation units are fully paralleled and distributed. In addition, $Q=GG^T$ is precomputed to reduce the complexity, and finding new coefficients $F_i$ can be independently computed and distributed. Remarkably, the most heavy computation of $Y=FV^T$ and $H=FF^T$ is divided into computing $F_iV_i^T$ and $F_iF_i^T$ to be parallel and distributed.

Particularly, the distributed system using MapReduce is described in Fig. \ref{fig:DisSysDia}. In this computing model, data instances and the instance projection are parallel and distributed over the Map nodes. The Reduce nodes sum up the results $F_iF_i^T$ and $F_iV_i^T$ of the Map nodes.  Subsequently, in the M-step, the results $FF^T$ and $FV^T$ are employed to compute latent components $G$. This M-step computation can be conducted by a single machine or a distributed system, which depends on the dimension of problem because the time to distribute this computation over the network is usually considerable. 

In comparison with the previous algorithms, this computing model is much more effective than the previous models~\cite{Gemulla2011,Liu2010,Sun2010Large} by the following reasons:

\indent $\bullet$ The necessary memory used in computing nodes is $\mathcal{O}(size(G, Y, H))$ = $\mathcal{O}(r(r+n))$. The necessary memory used in the controlled node is $\mathcal{O}(size(G, Y, $ $H, Q))$ = $\mathcal{O}(r(r+n))$. In practice, approximate solutions of NQP problems should be cached in hard disks in order to increase accuracy and reduce the number of iterations.

\indent $\bullet$ At each distributed iteration, the computation is fully decomposed into basic computations units, which enhances the convergence speed to the optimal solution because the size of optimization is significantly reduced. Furthermore, the expensive computation $FV^T$ and $FF^T$ is fully parallelized and distributed over the computing nodes.

\indent $\bullet$ The computational model is conveniently implemented because computing NMF model is divided into basic computation units as NQP problems that are independently solved, and the optimization is carried out on vectors instead of matrices.

In the next section, we propose a novel algorithm, Algorithm~\ref{algo:NQP},  to solve approximately NQP problems, which is robust and effective because it only uses the first derivative  and does not consider the ill-condition of matrix inverse.

\subsection{Fast algorithm for nonnegative quadratic programming}

In this section, we briefly review the literature before proposing the novel algorithm to solve NQP Problem~\ref{eq:QP} for real large-scale NMF applications.

Regarding algorithms for NNLS and its equivalent problem NQP, numerous algorithms are proposed to deal with high dimension~\cite{Chen2009Non}. Generally, methods for solving NNLS can divided into two groups: active-set and iterative methods~\cite{Chen2009Non}. Active-set methods are traditional to solve accurately~\cite{Bro1997,Lawson1974}. However, they require heavy computation in repeatedly computing $(A^TA)^{-1}$ with different set of passive variables. Hence, iterative methods that can handle multiple active constraints in each iteration have more potential for fast NMF algorithms~\cite{Chen2009Non,Kim2006NNLS,Kim2013}. Hence, iterative methods can deal with more large-scale problems. Among the fast  iterative methods, the coordinate descent method~\cite{Franc2005seq} has fast approximation, but has the zip-zag problem when the solution requires high accuracy.  In addition, accelerated  methods~\cite{Nesterov1983} 
has a fast convergence $\mathcal{O}(1/k^2)$~\cite{Guan2012}, which only require the first order derivative. However, one major disadvantage of the methods is that they require a big number of iterations because their step size is limited by $\frac{1}{M}$ that can  be very small for large-scale problems; where $M$ is Lipschitz constant. More recently, the anti-lopsided algorithm~\cite{Nguyen2015} re-scale variables to obtain a linear convergence in the sub-space of passive variables. Unfortunately, the passive variables are unknown in advance, so several iterations are required to determine them. In addition, the complexity of each iteration is considerable about $\mathcal{O}(r^2)$.

Therefore, we propose a combinational algorithm of  the anti-lopsided algorithm~\cite{Nguyen2015} and the greedy coordinate block descent algorithm~\cite{Hsieh2011} to reduce the number of iterations as well as complexity. Particularly, the proposed algorithm, Algorithm~\ref{algo:NQP}, contains two main steps: The first step, from Line~\ref{lst:line:rescale1} to Line~\ref{lst:line:rescale2}, rescales variables to avoid rescaling problems of the first order methods by replacing $y = x.*\sqrt{diag(H)}$, we have:

\begin{equation}
f(x) = \frac{1}{2}x^THx + h^Tx = \frac{1}{2}y^TQy + q^Ty
\end{equation}

\noindent where $Q=\frac{H}{\sqrt{diag(H)diag(H)^T}}$ and $q=\frac{h}{\sqrt{diag(H)}}$ such that $\frac{\partial^2f}{\partial^2 y_i} = Q_{ii} = \frac{H_{ii}}{\sqrt{H_{ii}H_{ii}}} =1$ for $\forall i$. By the way, the rate of change of a quantity through variables equals to a constant and the exact line search can converge at an exponential rate of $(1-\frac{1}{||Q||_2})^r$~\cite{Nguyen2015} in the sub-space of passive variables.  The passive variables are variables belongs the set $P = \{x_i | x_i > 0\ or\ \nabla f_i < 0\}$ that changes through iterations.

The second step contains a loop of iterations, from Line~\ref{lst:line:loop1} to Line~\ref{lst:line:loop2}, each of  which is clearly divided into two parts: one part from Line~\ref{lst:line:anti1} to Line~\ref{lst:line:anti2} inherited from the anti-lopsided algorithm~\cite{Nguyen2015} and the other from Line~\ref{lst:line:coordinate1} to Line~\ref{lst:line:coordinate2} based on the fast coordinate descent algorithm~\cite{Hsieh2011}. The anti-lopsided algorithm guarantees the linear convergence $(1-\frac{1}{||Q||^2_2})^k\ (||Q||^2_2 \leq r)$ in the sub-space of passive variables to avoid the zip-zag problem of the fast coordinate descent algorithm, while the coordinate block descent algorithm speeds up the convergence to the final optimal set of passive variables. In addition, the complexity of each part is still kept in $\mathcal{O}(r^2)$. As a result, the proposed algorithm will utilize advantages of both algorithms to attain a fast convergence, while retaining the same low complexity $\mathcal{O}(r^2)$ of each iteration.

To comprehend the proposed algorithm's effectiveness, we consider optimizing  Function~\ref{eq:NQP}:
\begin{equation}
\label{eq:NQP}
f(x) = \frac{1}{2}x^T\begin{bmatrix} 1&0.1\\ 0.1&10 \end{bmatrix}x + [-80 -100]x
\end{equation}

The exact search gradient algorithm,  from Line~\ref{lst:line:anti1} to Line~\ref{lst:line:anti2}, starting with $x_0 = [200\ 20]^T$ performs 59 iterations to reach the optimal solution, see Fig.~\ref{fig:not_using_scissored}. However, the proposed algorithm only needs 1 iterations to reach the optimal solution, see Fig.~\ref{fig:using_antilopsided_scissored} because we optimize Function~\ref{eq:NQPI} instead of Function~\ref{eq:NQP}; where Function~\ref{eq:NQPI} is equivalently obtained by applying the steps from Line~\ref{lst:line:anti1} to Line~\ref{lst:line:anti2}. The exact search gradient algorithm becomes much faster because the shape of Function~\ref{eq:NQPI} become more sphere, and its derivative is more effective to optimize the objective function. 
\begin{equation}
\label{eq:NQPI}
f(y) = \frac{1}{2}y^T\begin{bmatrix} 1&\frac{0.1}{\sqrt{10}}\\ \frac{0.1}{\sqrt{10}}&1 \end{bmatrix}y + [\frac{-80}{\sqrt{10}}\ \  \frac{-100}{\sqrt{10}}]y
\end{equation}

\begin{algorithm2e} 
	\caption{Fast Combinational Algorithm for NQP} \label{algo:NQP}
	\KwIn{$H \in R^{r \times r}$ and $h \in R^r$ and $x_0$}
	\KwOut{$x$ minimizing $\frac{1}{2}x^THx + h^Tx$\\
		\ \ \ \ \ \ \ \ \ \ \ \ \ subject to: $x \succeq 0$}
	\Begin{
		/*Having a variable $maxStop = 0$ for each thread of computation */\;
		/*Re-scaling variables*/\;
		$Q = \frac{H}{\sqrt{diag(H)diag(H)^T}}$\; \label{lst:line:rescale1}
		$q = \frac{h}{\sqrt{diag(H)}}$\; 
		/*$\text{Solving NQP: minimizing} f(x) = \frac{1}{2}x^TQx + q^Tx$*/\; 
		$x = x_0.*\sqrt{diag(H)}$\;\label{lst:line:rescale2}
		$\nabla f = Qx + q$\;
		\Repeat{ $(||\tilde{f}_k||^2_2 \leq \epsilon||\tilde{f}_0||^2_2) \text{ or } (||\tilde{f}_k||^2_2  \leq maxStop)$}{ \label{lst:line:loop1}
			/*Exact Line Search*/\;
			$\nabla \bar{f} = \nabla f[x > 0 \text{ or } \nabla f < 0]$\; \label{lst:line:anti1}
			$\alpha = \argmin{\alpha} f(x_k - \alpha \nabla \bar{f}) = \frac{||\nabla \bar{f}||^2_2}{\nabla \bar{f}^TQ\nabla \bar{f}}$\;
			$x_{k} = [x_{k-1} - \alpha \nabla \bar{f}]_+$\;
			$\nabla f_k = \nabla f_k + Q(x_{k} - x_{k-1})$\; \label{lst:line:anti2}
			/*Block Coordinate Descent*/\;
			\For{t=1 to n}{  \label{lst:line:coordinate1}
				$\triangle x_i = max(0, [x_{k}]_i - \frac{f_i}{Q_{ii}}) - [x_{k}]_i\ \forall i$\;
				$p = \argmax{i} |f(x_k) - f(x_k + \triangle x_i)|$\;
				$\nabla f_k = \nabla f_k + Q_p \triangle x_p$\;
				$[x_{k}]_p = [x_{k}]_p + \triangle x_p$\; \label{lst:line:coordinate2}
			}
		}\label{lst:line:loop2}
		$maxStop = max(maxStop, ||\tilde{f}_k||^2_2)$\;
		\Return{$\frac{x_k}{\sqrt{diag(H)}}$} 
	}
\end{algorithm2e}

\begin{figure}
		\centering
		\includegraphics[scale=0.5]{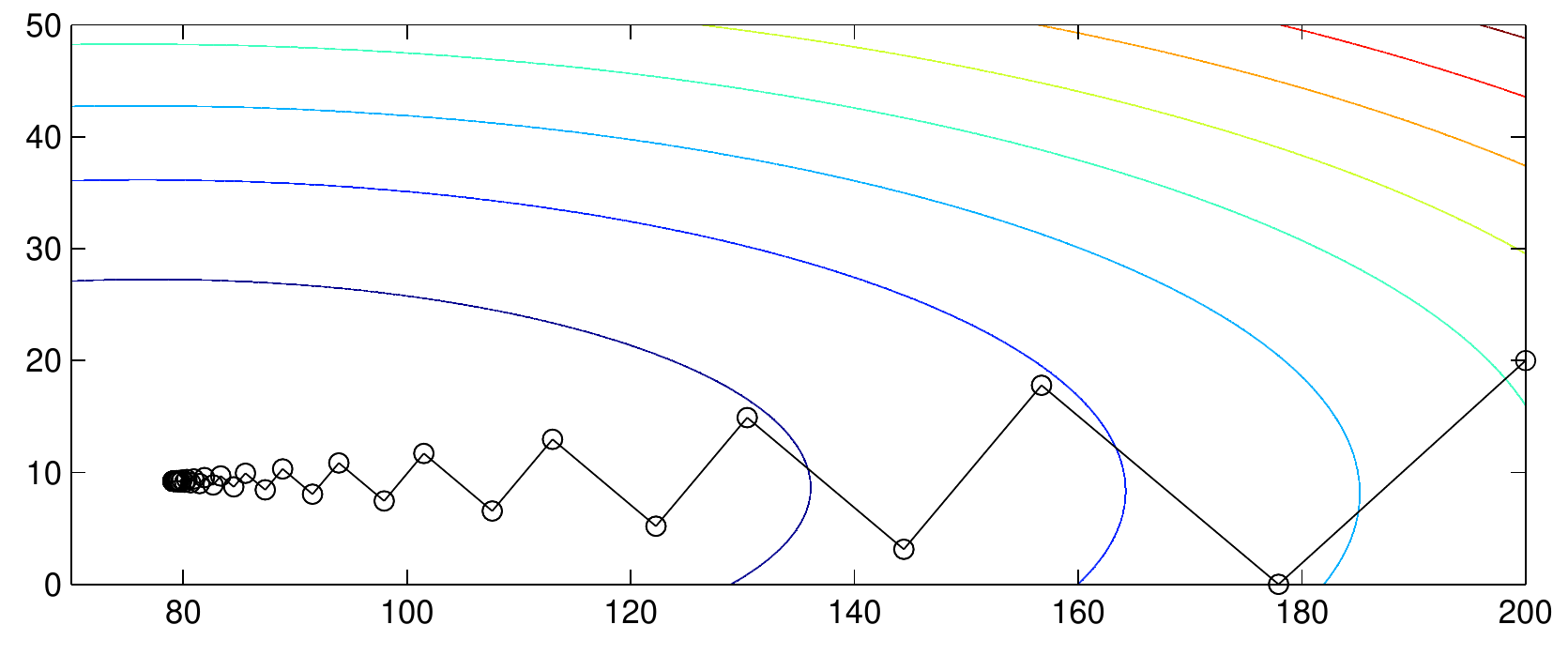}
		\caption{59 optimizing steps in iterative exact line search method using  the first order derivative for the function~\ref{eq:NQP} starting at $x_0=[200\ 20]^T$}
		\label{fig:not_using_scissored}
 \end{figure}
 
\begin{figure}
	\centering
		\includegraphics[scale=0.5]{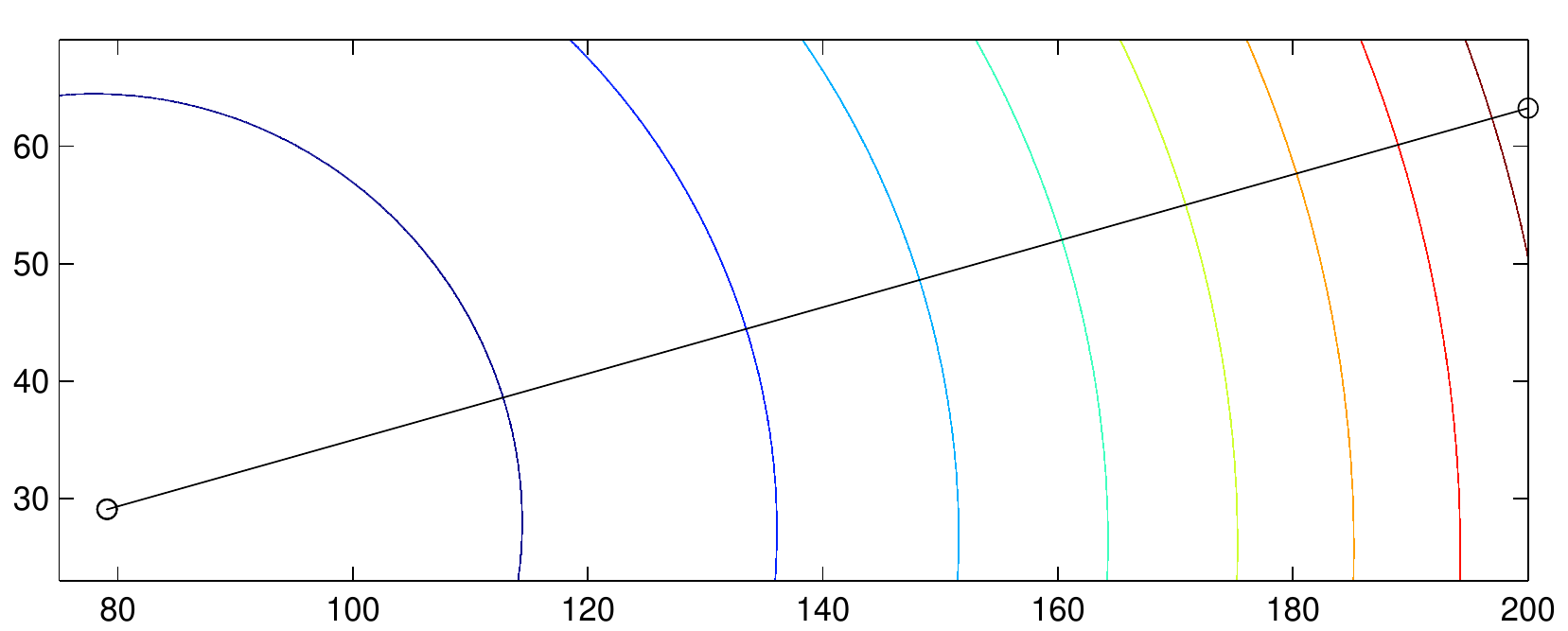}
		\caption{1 optimizing steps in iterative exact line search method using  the first order derivative for the function~\ref{eq:NQPI} starting at $y_0 = x_0\sqrt{diag(H)}$}
		\label{fig:using_antilopsided_scissored}
\end{figure}

Moreover, Algorithm~\ref{algo:NQP} only attains approximate solutions because achieving the optimal solution is controversial for the reasons that its computation is expensive and it can leads to the zig-zag problem in optimizing a non-convex function. In addition, it is necessary to control and balance the quality of the convergence to the optimal solution. Hence, we employ an accelerated condition $(||\tilde{f}_k||^2_2 \leq \epsilon||\tilde{f}_0||^2_2)$ to regulate the quality of the convergence to the optimal solutions of the NQP problems in comparison with initial values and  a fast-break condition $(||\tilde{f}_k||^2_2  \leq maxStop)$ to balance the quality of the convergence among variables in each thread of the computation. As a result, the objective function converges faster through iterations; and the complexity and average of the iteration number are reduced significantly. 

\subsection{Extensions for $L_1$ $L_2$ regularized NMF}

In this section, we consider solutions for $L_1$ $L_2$ regularized NMF variants to control the quality of NMF. $L_1$ regularized NMF~\cite{Hoyer2002Non} aims to achieve sparse solutions in optimization problems. Usually, only the coefficient matrix $F$ is penalized to control its sparsity. Meanwhile, concerning $L_2$ regularized NMF, the penalty terms of $F$ and $G$ are added to control smoothness of solutions in NMF~\cite{Pauca2006}. Fortunately, the objective functions of $L_1$ $L_2$ regularized NMF can be turned into NQP problems, of which solutions are completely similar to the general NMF. Particularly, in the most general variant, the objective function $J(X||GF)$ is formulated by:

\begin{equation}
||X-GF||^2_2 + \mu_1 ||F||_1 + \beta_1 ||G||_1 + \mu_2 ||F||_2^2 + \beta_2 ||G||_2^2
\end{equation}

\noindent where $||.||_1$ is the $L_1$-norm,  $||.||_2$ is the $L_2$-norm, and $\mu_1, \mu_2, \beta_1, \beta_2$ are regularized parameters that tradeoff the sparsity and the smoothness of the information loss. Obviously, both the E-step and  the M-step need to solve the same NNLS problems when one of the two matrices is fixed. For example, in a E-step, we can minimize the objective function by independently solving NQP problems  when fixing $G$:

\begin{equation}\label{eq:L1L2NMF}
\begin{aligned}
&J(X||GF) =  \frac{1}{2}||X-GF||^2_2 + \mu_1 ||F||_1 + \mu_2 ||F||^2_2 + \text{C}\\
& = \sum_{m=1}^{M} ( \frac{1}{2}||X_m-GF_m||^2_2 + \mu_1 (1^K)^TF_m + \mu_2 F_m^TIF_m) + \text{C}\\
&= \sum_{m=1}^{M} ( \frac{1}{2}F_m^TQF_m + q^TF_m) + \text{C}\\
\end{aligned}
\end{equation}

\noindent where $Q=G^TG+ 2\mu_1 I$, $q^T= - X_m^TG^T + \mu_2 1^K$ and $\text{C}$ is a constant.

This transformation from minimizing the objective functions into solving NQP problems independently is comprehensive to understand and simplify the variants of NMF problems as much as possible. As a result, we can conveniently implement NMF and its $L_1$ $L_2$ regularized variants in parallel distributed systems as in sub-section \ref{subsec:ParallelandDistributed}.

In comparison with the previous algorithms that optimizing the objective function works on the whole of matrices, this approach decomposing the objective function is easier to parallelize and distribute the computation. Additionally, it is faster to reach the solutions because it only performs on a smaller set of variables.

\section{Theoretical Analysis}
\label{sec:Complexity}

In this section, we investigate the convergence of Algorithm~\ref{algo:NQP} and the complexity of Algorithm~\ref{algo:Parallel} using
Algorithm~\ref{algo:NQP} 

\subsection{Convergence}

In this section, we only consider the convergence rate of Algorithm~\ref{algo:NQP} by the general NMF for the two following reasons. Firstly, $L_1$ regularized coefficients do not affect on the complexity. Secondly, $L_2$ regularized coefficients are often small, and they change Lipschitz constants $m$ and $M$ by adding a small positive value, where $m$ and $M$ are positive Lipschitz constants of strongly convex function $f(x)$ satisfying $mI \preceq \frac{\partial^2 f}{\partial^2x} \preceq MI$ and $I$ is the identity matrix. Hence,  $L_2$ regularized coefficients slightly change the convergence rate because it depends on $\frac{m}{M}$.

Based on~\cite{Nguyen2015}, consider the complexity of  Algorithm~\ref{algo:NQP}, we have:

\begin{theorem}
	Algorithm~\ref{algo:NQP} linearly converges at the rate of $\mathcal{O}(1-\frac{1}{||Q||_2})^k$ in the sub-space of passive variables, where $\sqrt{r} \leq ||Q||_2 \leq r$, $r$ is the dimension of solutions or the number of latent factors, and $k$ is the number of iterations.
\end{theorem}
\begin{proof}
	
	From \cite{Nguyen2015}, we have:
	
	\textbf{Remark 1:} After $(k+1)$ iterations, $f(x^{k+1}) - f^* \leq (1 - \frac{m}{M})^k (f(x^0) - f^*)$, where $m I \preceq \nabla^2 f \preceq MI$, $f^*$ is the minimum value of $f(x)$, and $f(x)$ is a strongly convex function of the passive variables.
	
	We have $\nabla^2 f = Q$, and
	
	$x^TIx \leq \sum_{i=1}^{r}\sum_{j=1}^{r}Q_{ij}x_ix_j=x^TQx$ since $x \geq 0, Q\geq 0$, and $Q_{ii}=1$.
	$\Rightarrow I \preceq Q$.
	
	Moreover, based on Cauchy-Schwarz inequality, we have:
	\begin{equation*}
	\begin{aligned}
	\ &(\sum_{i=1}^{r}\sum_{j=1}^{r}Q_{ij}x_ix_j)^2 \leq (\sum_{i=1}^{r}\sum_{j=1}^{r} Q_{ij}^2) (\sum_{i=1}^{r}\sum_{j=1}^{r} ({x_i}{x_j})^2) \\
	\Rightarrow &\sum_{i=1}^{r}\sum_{j=1}^{r}Q_{ij}x_ix_j \leq \sqrt{||Q||^2_2 (\sum_{i=1}^{r} {x_i}^2)^2}\\
	\Leftrightarrow & x^TQx \leq ||Q||_2x^TIx \ \ (\forall x) \ \ \Leftrightarrow Q \preceq ||Q||_2I\\
	\end{aligned}
	\end{equation*}
	
	Finally, $\sqrt{r} = \sqrt{\sum_{i=1}^{r} Q_{ii}^2} \leq ||Q||_2 = \sqrt{\sum_{i=1}^{r}\sum_{j=1}^{r}Q_{ij}^2} \leq \sqrt{r^2} = r$ since $-1 \leq Q_{ij} = \cos(H_i, H_j) \leq 1$. Therefore, we have:
	
	\textbf{Remark 2:} $I \preceq \nabla^2 f = Q \preceq ||Q||_2I$; where $\sqrt{r} \leq ||Q||_2 \leq r$. 
	
	From Remark 2 setting $m=1$ and $M=||Q||_2 \leq r$, and Remark 1, we have Theorem~\ref{theo:Complexity}.
\end{proof}

Actually, the exact line search step, from Line~\ref{lst:line:anti1} to Line~\ref{lst:line:anti2} in  Algorithm~\ref{algo:NQP}, guarantees linear convergence of $\mathcal{O}(1-\frac{1}{||Q||_2})^k$ in the sub-space of passive variables. However, the set of passive variables changes through iterations. Hence,  we employ the fast block coordinate descent steps, from Line~\ref{lst:line:coordinate1} to Line~\ref{lst:line:coordinate2} in  Algorithm~\ref{algo:NQP}, that rapidly restrict the domain of solution to converge to the final optimal sub-space of passive variables of the solution. Therefore, the proposed algorithm linearly converges  and requires very few iterations.

\subsection{Complexity}

In this section, we analyze the complexity of Algorithm~\ref{algo:Parallel} using Algorithm~\ref{algo:NQP} to solve NQP problems. If we assume that the complexity for each iteration contains
$\mathcal{O}(nr^2)$ in computing $Q=G^TG$ , $\mathcal{O}(mnr)$ in computing $Y=VF^T$, $\mathcal{O}(mr^2)$ in computing $H=FF^T$,  $\mathcal{O}(\overline{k}mr^2)$ in computing $F$ and $\mathcal{O}(\overline{k}nr^2)$ in computing $G$, where $\overline{k}$ is the number of iterations, then we have the following Lemma~3:

\begin{theorem}
	\label{theo:Complexity}
	The complexity of each iteration in  Algorithm~\ref{algo:Parallel} using Algorithm~\ref{algo:NQP} to solve NQP problems is $\mathcal{O}((m+n)r^2+mnr+\overline{k}(m+n)r^2)$. In addition, it is $\mathcal{O}((m+n)r^2+rS(mn)+\overline{k}(m+n)r^2)$ for sparse data, where $S(mn)$ is the number of non-zero elements in data matrix $V$. 
\end{theorem}

Theorem~\ref{theo:Complexity}  is significant for big data, because the data is usually big and sparse. In other words, $mn$ is actually large, but $S(mn)$ is small; so $mn \gg (m+n)r^2 + rS(mn) + \bar{k}(m+n)r^2$. Hence, in experimental evaluation Section~\ref{sec:Experimental}, we prove that our algorithm can run on large high-dimension sparse datasets such as Nytimes for an acceptable time. In that dataset, $mnr \gg rS(mn) \gg (m+n)r^2$, so the running time $T(m,n,r) \approx rS(mn)$ since $m, n \gg r$. 

Moreover, Table~\ref{tab:complexity} shows a comparison of the complexity in an iteration of our proposed algorithms (Alo) with other state-of-the-art algorithms' in the literature: Multiplicative Update Rule (MUR)~\cite{Lee2001}, Projected Nonnegative Least Squares (PrN)~\cite{Berry2007}, Projected Gradient (PrG)~\cite{Lin2007Pro}, Projected Quasi-Newton (PQN)~\cite{Zdunek2006}, Active Set (AcS)~\cite{Kim2008}, Block Principal Pivoting (BlP)~\cite{Kim2008Toward},  Accelerated Hierarchical Alternating Least Squares (AcH) , Fast Coordinate Descent Methods with Variable Selection (FCD)~\cite{Hsieh2011}, and Nesterov's Optimal Gradient Method (Nev)~\cite{Guan2012}. It can be seen that the complexity of our proposed algorithm is highly comparable with that of other algorithms, and the speed of algorithms depend on the number of iterations. In the experimental evaluation, we will show that the iteration number of our algorithm is highly competitive with other algorithms'. Remarkably, moreover, our proposed algorithm has the following properties that other algorithms has yet considered:

\begin{itemize}
	\item Exploit the sparseness of datasets,
	\item Runnable for big datasets in limited internal memory  systems,
	\item Convenient to implement in fully paralleled and distributed systems.
\end{itemize}

\begin{table}
	\centering
	\caption{Complexity of an iteration in NMF solvers} 
	\label{tab:complexity} 
	\begin{tabular}[1.0]{ll} 
		\hline\noalign{\smallskip}
		\textbf{Solver} & \ \ \ \ \ \ \textbf{Complexity ($\mathcal{O}$)}\\ 
		\noalign{\smallskip}\hline\noalign{\smallskip}
		MUR~\cite{Lee2001} & $mnr + (m+n)r^2$ \\ 
		PrN~\cite{Berry2007}  & $mnr + (m+n)r^2 + r^3$ \\
		PrG~\cite{Lin2007Pro} & $(m+n)r^2+rmn+\overline{k}\overline{t}(m+n)r^2$ \\ 
		PQN~\cite{Zdunek2006} & $\overline{k}(mnr + m^3r^3 + n^3r^3)$\\ 
		BlP~\cite{Kim2008Toward} & $(m+n)r^2+mnr+\overline{k}(m+n)r^2$ \\ 
		AcS~\cite{Kim2008} & $(m+n)r^2+rmn+\overline{k}(m+n)r^2$ \\ 
		FCD~\cite{Hsieh2011} & $(m+n)r^2+rS(mn)+\overline{k}(m+n)r^2$\\
		AcH~\cite{Gillis2012HALS} & $(m+n)r^2+rS(mn)+\overline{k}(m+n)r^2$\\
		Nev~\cite{Guan2012} & $(m+n)r^2+mnr+\overline{k}(m+n)r^2$ \\ 
		Alo & $(m+n)r^2+rS(mn)+\overline{k}(m+n)r^2$ \\ 
		\noalign{\smallskip}\hline
	\end{tabular}
	\begin{tablenotes}
		\small
		\noindent where $m, n$ is the matrix size, $r$ is the number of latent components, $\overline{k}$ is the average number of iterations, $\overline{t}$ is the average number of internal iterations, and $S(mn)$ is the number of non-zero elements of data matrix $V$. 
		To easily compare among the algorithms, we consider $r$ update times for Algorithm FCD as one iteration because the complexity of one update is $\mathcal{O}(r)$, while the complexity of one iteration in other accelerated algorithms is $\mathcal{O}(r^2)$.
	\end{tablenotes}
\end{table}

\section{Experimental evaluation}
\label{sec:Experimental}

In this section, we investigate the effectiveness of the proposed algorithm \textbf{Alo} by comparing it to 7 carefully selected state-of-the-art NMF solvers belongs to different approaches:
\begin{itemize}
	\item \textbf{MUR}: Multiplicative Update Rule~\cite{Lee1999},\\
	\vspace{-0.4cm}
	\item \textbf{PrG}: Projected Gradient Methods~\cite{Lin2007Pro},\\
	\vspace{-0.4cm}
	\item \textbf{BlP}:  Block Principal Pivoting method~\cite{Kim2008Toward},\\
	\vspace{-0.4cm}
	\item \textbf{AcS}: Fast Active-set-like method~\cite{Kim2008},\\
	\vspace{-0.4cm}
	\item \textbf{FCD}: Fast Coordinate Descent methods with variable selection~\cite{Hsieh2011},\\
	\vspace{-0.4cm}
	\item \textbf{AcH}:  Accelerated Hierarchical Alternating Least Squares~\cite{Gillis2012HALS},\\
	\vspace{-0.4cm}
	\item \textbf{Nev}: Nesterov's optimal gradient method~\cite{Guan2012}.
\end{itemize}

\textbf{Test cases}: In this experiment, we design two tests using four datasets shown in Table \ref{tab:data_sets}. In the first test, 3 typical datasets with different sizes are used:
Faces\footnote{\url{http://cbcl.mit.edu/cbcl/software-datasets/FaceData.html}},
Digits\footnote{\url{http://yann.lecun.com/exdb/mnist/}}  and Tiny Images~\footnote{\url{http://horatio.cs.nyu.edu/mit/tiny/data/index.html}}.  
For these tests, the algorithms are compared in terms of convergence, optimality, and average of the iteration number to investigate their performance and effectiveness. Additionally, average of the the iteration number $\bar{k}$ for approximate solutions of the sub-problems as NNLS or NQP is to compare the complexity of algorithms. In the second test, a large dataset containing tf-idf values computed from the text dataset Nytimes\footnote{\url{https://archive.ics.uci.edu/ml/machine-learning-databases/bag-of-words/}} is used to verify the performance and the feasibility of our parallel algorithms on sparse large datasets. 

\begin{table}
	\centering
	\caption{Dataset Information} 
	\label{tab:data_sets}
	\begin{tabular}{lcccc}
		\hline\noalign{\smallskip}
		\textbf{Data-sets} & \textbf{$m$} & \textbf{$n$}& $r$  & {MaxIter}\\ 
		\noalign{\smallskip}\hline\noalign{\smallskip}
		Faces & $6977$ & 361 & 60 & 300 \\ 
		Digits & $6.10^4$ & 784 & 80 & 300 \\ 
		Tiny Images  & $5.10^4$ & 3,072 & 100 & 300 \\ 
		Nytimes & $3.10^5$ & 102,660& 100,...,200 & 300 \\ 
		\noalign{\smallskip}\hline
	\end{tabular}
\end{table}

\textbf{Environment settings}: To be fair in comparison, for the first test, the programs of compared algorithms are written in the same language Matlab 2013b, run by the same computer Mac Pro 8-Core Intel Xeon E5 3 GHz RAM 32 GB, and initialized by the same factor matrices $G_0$ and $F_0$. The maximum number of threads is set to 10 while keeping 2 threads for other tasks in the operation system. For the second test, the proposed algorithm is written in Java programming language to utilize the data sparseness. 

\textbf{Source code}: The source codes of \textbf{MUR}, \textbf{PrG}, \textbf{BlP}, \textbf{AcS}, \textbf{FCD}, \textbf{AcH}, and \textbf{Nev} are downloaded from \footnote{\url{http://www.cs.toronto.edu/~dross/code/nnmf.m}}, \footnote{\url{https://github.com/kimjingu/nonnegfac-matlab}}, \footnote{\url{http://www.csie.ntu.edu.tw/~cjlin/nmf/}}, \footnote{\url{http://dl.dropboxusercontent.com/u/1609292/Acc\_MU\_HALS\_PG.zip}}, and \footnote{\url{https://sites.google.com/site/nmfsolvers/}}. For convenient comparison in the future, we publish all the source codes and datasets in \footnote{\url{https://bitbucket.org/[aaa-zzz]/alnmf}}.

\subsection{Convergence}

In this experiment, we investigate the convergence of algorithms by information loss $\frac{1}{2}||X-GF||^2_2$ in terms of time and the iteration number. In terms of time, see Fig.~\ref{fig:convergence_time}, the proposed algorithm Alo is remarkably faster than the other algorithms for the three different-size datasets: Faces, Digits and Tiny Images. Especially, for the largest dataset Tiny Images, the distinction between the proposed algorithm and the runner-up algorithm AcH is easily recognized. Furthermore,  in terms of the iteration number, see Fig.~\ref{fig:convergence_iterators}, the proposed algorithm converges to the stationary point of solutions faster than the others. This observation is clear for large datasets as Digits and Tiny Images. The results are significant in learning NMF models for big data because the proposed algorithm not only converges faster but also uses a less number of iterations, and the time of reading and optimization through a big dataset is actually considerable.

\begin{figure}
	\includegraphics[scale=1.2]{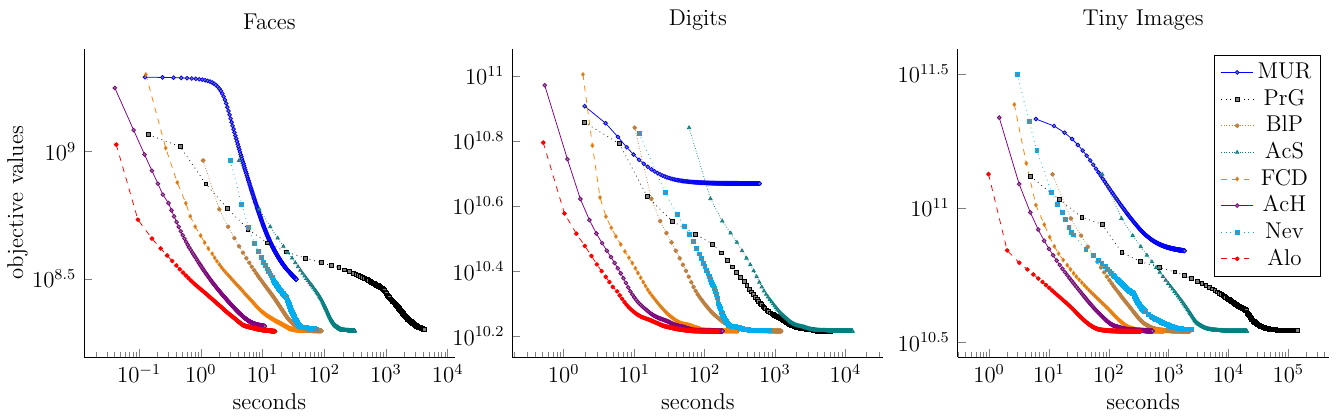}
	\caption{Objective function values $||V-GF||^2_2/2$ versus CPU seconds for datasets: Faces, Digits, and Tiny Images}
	\label{fig:convergence_time}
\end{figure}

\begin{figure}
	\includegraphics[scale=1.2]{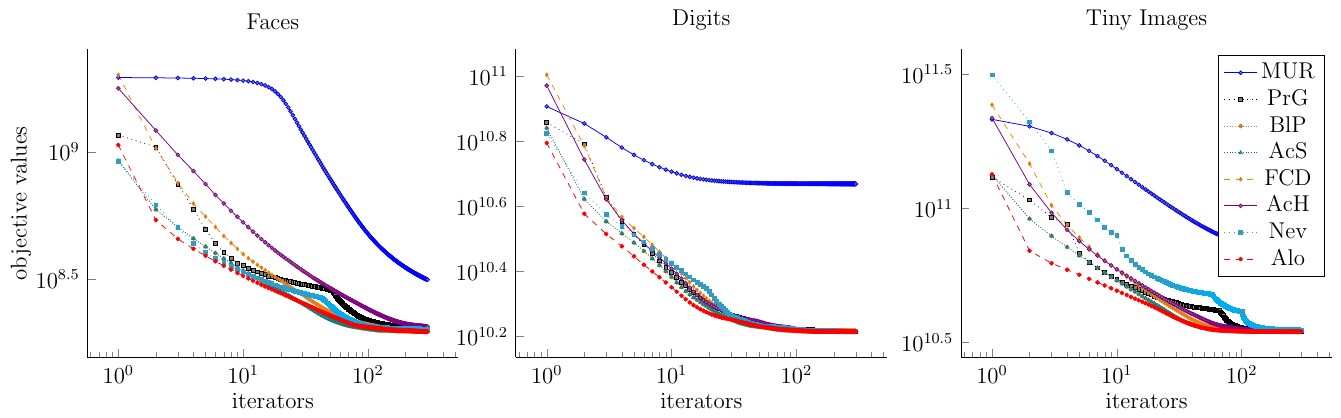}
	\caption{Objective function values $||V-GF||^2_2/2$ in terms of the iteration number for datasets: Faces, Digits, and Tiny Images}
	\label{fig:convergence_iterators}
\end{figure}

\subsection{Optimality}

After more a decade of rapid development, numerous algorithms have been proposed for solving NMF as a fundamental problem in dimension reduction and learning representation. Currently, the difference of the final loss information $||V - WH||^2_2$ among the state-of-the-art methods is inconsiderable in comparison to the square of information $||V||^2_2$. However, the small difference represents the effectiveness of the optimization methods because NMF algorithms often slowly converge when the approximate solution is close to the optimal local solution. Hence, in Table~\ref{tab:optimal}, the final values of the objective function $\frac{1}{2}||V-WH||^2_2$ investigate the optimality and the effectiveness of the optimization methods. Noticeably, Algorithm AcH fast converges over time and has a low average of the iteration number, but it has the optimal values much higher than the proposed algorithm because it uses a time-break technique to interrupt the optimization algorithm. In addition, the proposed algorithm achieves the best optimality for all three datasets. This result additionally represents the robustness of the proposed method, which is highly competitive with the state-of-the-art methods.

\begin{table}
	\centering
	\caption{Optimal Values of NMF solvers} \label{tab:optimal} 
	\begin{tabular}{lcccccccc}
		\hline\noalign{\smallskip}
		Dataset & MUR & PrG & BlP & AcS & FCD & AcH & Nev & Alo\\
		\noalign{\smallskip}\hline\noalign{\smallskip}
		Faces ($\times 10^{8}$) & 3.142 & 2.003 & 1.975 & 1.975 & 1.983 & 2.058 & 2.003 & \textbf{1.966} \\
		Digits ($\times 10^{10}$) & 4.659 & 1.639 & 1.641 & 1.641 & 1.644 & 1.640 & 1.646 & \textbf{1.638} \\
		Tiny Images ($\times 10^{10}$) & 6.925 & 3.483 & 3.472 & 3.472 & 3.474 & 3.484 & 3.513 & \textbf{3.468} \\
		\noalign{\smallskip}\hline
	\end{tabular}
\end{table}

\subsection{Average of iteration number}

In this section, we investigate the complexity of the NMF solvers by average of the iteration number $\bar{k} = \frac{\text{number of internal iteration}}{\text{MaxIter}(m + n)}$ for approximate solutions of sub-problems as NNLS or NQP because the complexity of algorithms mainly depends on this number, see Table~\ref{tab:complexity}. Except for the original algorithm MUR with one update having the worst result, the proposed algorithm Alo employs at least average of the iteration number, see Table~\ref{tab:iterations}, especially for large datasets. In addition, the proposed algorithm does not employ any tricks to timely interrupt before one of the stopping conditions is satisfied, while the highly competitive algorithm AcH uses. Therefore, this result clearly represents the fast convergence of Algorithm~\ref{algo:NQP} as it is verified by a large number of NQP problems.

\begin{table}
	\centering
	\caption{Average of Iteration Number $\bar{k}$} \label{tab:iterations} 
	\begin{tabular}{lcccccccc}
		\hline\noalign{\smallskip}
		
		Dataset & MUR & PrG & BlP & AcS & FCD & AcH & Nev & Alo\\
		
		\noalign{\smallskip}\hline\noalign{\smallskip}
		
		Faces & 1.00 & 321.12 & 1116.96 & 102.09 & 1.54 & \textbf{1.11} & 29.21 & 1.29 \\
		Digits & 1.00 & 36.70 & 12503.75 & 305.94 & \textbf{1.00} & 1.05 & 23.36 & \textbf{1.00} \\
		Tiny Images & 1.00 & 767.45 & 12869.12 & 1086.51 & 1.38 & 2.52 & 29.32 & \textbf{1.27} \\
		
		\noalign{\smallskip}\hline
	\end{tabular}
\end{table}

\subsection{Running on large datasets}

In this section, we verify the feasibility of the proposed algorithm in learning NMF model for large datasets. Particularly, the proposed algorithm is implemented by Java programming language to exploit the data sparseness. Additionally, it runs on the large sparse text dataset Nytimes with different numbers of latent components, see Table~\ref{tab:data_sets}. Interestingly, the proposed algorithm can run with hundreds of latent components by a single computer in an acceptable time.

Fig.~\ref{fig:nytimes} shows  the performance of our algorithm running on the large sparse dataset Nytimes. Remarkably, the proposed algorithm only uses about 1 iteration on average to satisfy the accelerated condition of approximate solutions. Furthermore, the average of iteration time in learning NMF model linearly increases through the different numbers of latent components. This result totally fits the complexity analysis when $rnm \gg rS(mn) \gg (m+n)r^2 + \bar{k}(m+n)r^2$, so the complexity $T(m, n, r) \approx rS(mn)$ since $m, n \gg r$. Additionally, the objective function converges to the stationary point at about the $100^\text{th}$ iteration within the different numbers of latent components $r$, which is the same with the previous datasets. 

\begin{figure}
	\includegraphics[scale=1.2]{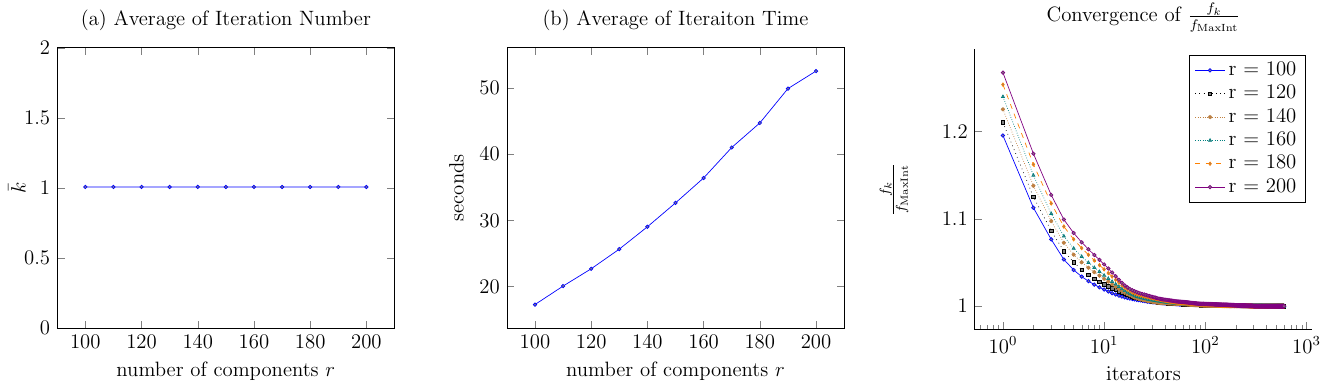}
	\caption{average of the iteration number $\bar{k}$,  average of iteration time, and convergence of $\frac{f_k}{f_\text{MaxInt}}$ in learning NMF model for the dataset Nytimes within the different numbers of latent components}
	\label{fig:nytimes}
\end{figure}

\subsection{Regularized NMF extensions}

In this section, we investigate the convergence of algorithms for regularized NMF extensions on three datasets: Faces, Digits, and Tiny Images. Due to the lack of available codes and the $L_1$ $L_2$ generalization of the other algorithms, only three algorithms AcS, Nev and Alo are compared within two regularized cases: $\mu_2=10^{-2}$ and $\mu_2=\beta_2=10^{-2}$, see Fig.~\ref{fig:regularised}. In comparison with other algorithms for regularized NMF extensions, the proposed algorithm Alo converges much faster than algorithms AcS and Nev.

\begin{figure}
	\includegraphics[scale=1.2]{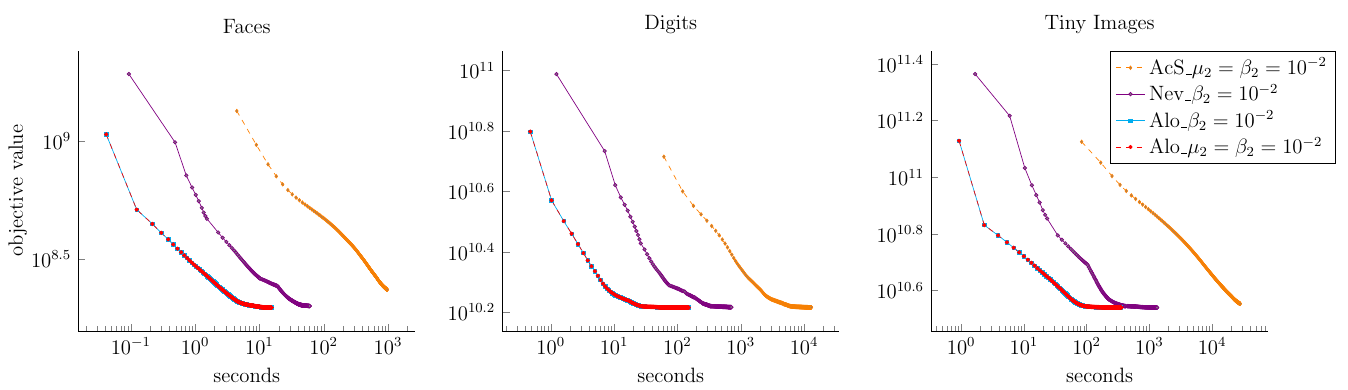}
	\caption{Convergence of regularized NMF Extensions for algorithms AcS, Nev and Alo within two regularized cases: $\mu_2=10^{-2}$ and $\mu_2=\beta_2=10^{-2}$}
	\label{fig:regularised}
\end{figure}

\section{Conclusion}
\label{sec:Conclusion}

In summary, our work has two major contributions:

Regarding nonnegative matrix factorization, we propose a flexible algorithm in an unified framework for NMF and its $L_1$ $L_2$ regularized variants based on full decomposition and a fast combinational algorithm of the anti-lopsided algorithm~\cite{Nguyen2015} and the greedy coordinate block descent algorithm~\cite{Hsieh2011}. The proposed algorithm has \textit{linear} convergence rate of $\mathcal{O}(1-\frac{1}{r})^k$ in optimizing each matrix factor in the sub-space of passive variables when fixing the other matrix, where $r$ is the number of latent components. The proposed algorithm is an advanced version of fast block coordinate descent methods and accelerated methods. In theory and practice, the proposed algorithm resolve some current major issues of NMF: fast learning algorithm, data sparseness exploit-ability, and parallel distributed feasibility using limited internal memory. Furthermore, the proposed algorithm flexibly adapts with all the variants of $L_1\ L_2$ NMF regularizations.

In experimental comparative evaluation, our algorithm overcomes 7 of the most art-the-state algorithms in large datasets about three significant aspects of convergence, average of the iteration number and optimality. In addition,  it can fully be parallelized and distributed because the computation using limited internal memory is decomposed into basic computation units as NQP problems. Concerning the feasibility in real applications, the proposed algorithm exploits the data sparseness to learn the huge sparse dataset Nytimes in an acceptable time by a single machine. Finally, the convergence of the proposed algorithm for $L_1 L_2$regularized NMF variants is much faster than that of the existing algorithms.

Concerning the optimization techniques for alternating least squares methods, we propose a fast algorithm, Algorithm~\ref{algo:NQP} for NQP problems, which not only has a linear convergence in theory but also is verified in practice about the three significant aspects by a large number of NQP problems conducted inside the NMF framework. Hence, we strongly believe that the algorithm can be effectively employed for alternating least square methods as the key problem in factorization methods.

\section*{Acknowledgement}
This work was supported by Asian Office of Aerospace R\&D under agreement number FA2386-13-1-4046; and 911 Scholarship from Vietnam Ministry of Education and Training.

\bibliography{NMFRef}   

\end{document}